\documentclass[]{amsart}

\usepackage[english]{babel}
\usepackage[utf8]{inputenc}
\usepackage{paralist}
\usepackage{amsmath, amsfonts, amssymb, amsthm}
\usepackage{color}
\usepackage{ulem}
\usepackage{tikz}
\allowdisplaybreaks

\newcommand{\CC}{\mathbb{C}}

\newcommand{\NN}{\mathbb{N}}
\newcommand{\PP}{\mathbb{P}}

\newcommand{\Hc}{\mathcal{H}}

\newcommand{\set}[1]{\left\{ #1 \right\}}
\newcommand{\setb}[1]{\left( #1 \right)}
\newcommand{\abs}[1]{\left| #1 \right|}

\newcommand{\gfr}{\mathfrak{g}}
\newcommand{\bino}[2]{\begin{pmatrix} #1 \\ #2 \end{pmatrix}}
\newcommand{\m}{\textup{m}}

\newtheorem{mymasterthm}{notForUse}
\theoremstyle{definition}

\theoremstyle{plain}
\newtheorem{mylemma}[mymasterthm]{Lemma}
\newtheorem{mythm}[mymasterthm]{Theorem}

\newtheorem{myprop}[mymasterthm]{Proposition}

\title{Perfect powers in polynomial power sums}

\author[C. Fuchs]{Clemens Fuchs}
\author[S. Heintze]{Sebastian Heintze}

\address{University of Salzburg\newline
	\indent Department of Mathematics\newline
	\indent Hellbrunnerstr. 34 \newline
	\indent A-5020 Salzburg, Austria}
\email{clemens.fuchs@sbg.ac.at, sebastian.heintze@sbg.ac.at}

\begin{document}
	
\maketitle

	\begin{abstract}
		We prove that a non-degenerate simple linear recurrence sequence $ (G_n(x))_{n=0}^{\infty} $ of polynomials satisfying some further conditions cannot contain arbitrary large powers of polynomials if the order of the sequence is at least two. In other words we will show that for $ m $ large enough there is no polynomial $ h(x) $ of degree $ \geq 2 $ such that $ (h(x))^m $ is an element of $ (G_n(x))_{n=0}^{\infty} $. The bound for $ m $ depends here only on the sequence $ (G_n(x))_{n=0}^{\infty} $.
		In the binary case we prove even more. We show that then there is a bound $ C $ on the index $ n $ of the sequence $ (G_n(x))_{n=0}^{\infty} $ such that only elements with index $ n \leq C $ can be a proper power.
	\end{abstract}
	
	\section{Introduction}
	
	An interesting question that was studied in several recent papers (e.g. cf. \cite{B,FK,FKK,FMZ,R,S,Z1,Z3,Z4} and also \cite{FP1,FP2,PT,FZ,FMZ}) is what one can say about the decomposition of complex polynomials (i.e. elements of the ring $\CC[x]$ of complex polynomials) regarding the composition operation.

	The invertible elements w.r.t. decomposition are the linear polynomials. We call $f(x)=g\circ h$ a non-trivial decomposition if neither $g$ nor $h$ is linear.
	We call $f(x)=g\circ h$ an $m$-decomposition if $\deg g=m$ and we say that $f$ is $m$-decomposable if an $m$-decomposition exists.
	We call $f$ indecomposable if $f$ admits only trivial decompositions.
	A pair $(g,h)$ is called equivalent to $(g',h')$ if there are $a,b\in\mathbb{C},a\neq 0$ such that $g(x)=g'(ax+b),$ $h(x)=(h'(x)-b)/a$.
	A pair $(g,h)$ is called cyclic if it is equivalent to $(g',x^m)$ and dihedral if it is equivalent to $(g'',T_m(x))$ where $(T_n)_{n=0}^\infty$ denotes the sequence of Chebyshev polynomials (defined by $T_n(x+1/x)=x^n+1/x^n$) and $g',g''\in\mathbb{C}[x]$.
	We call $f$ cyclic if it is equivalent to a polynomial $g$ with $g(x)=x^n$ for some $n>1$, and dihedral if it is equivalent to $g=T_n$ for some $n>2$; here equivalent means that there are linear polynomials $l_1,l_2$ such that $f=l_1\circ g\circ l_2$.
	Mainly, one is interested in non-trivial decompositions (with two factors, an {inner} and an {outer} factor) of polynomials with coefficients in $\mathbb{C}$.
	It is natural to restrict to a subset {$\mathcal{I}$} of $\mathbb{C}[x]$ which is described by a finite amount of data and then to ask whether or not all decompositions in this subset can be described {in finite terms} depending on the data describing the subset.
	
	In \cite{Z3} and \cite{Z1} Zannier studied such decompositions with special focus on the number $ l $ of terms of the polynomial $ f $.
	Let us consider for a moment lacunary polynomials, i.e. we set $\mathcal{I}=\{f\in\mathbb{C}[x]; f$ has at most $\ell$ non-constant terms$\}$.
	Motivated by previous work of Erd\H{o}s \cite{E} and Schinzel \cite{Sch}, Zannier \cite{Z1} finally proved that there are integers $p,J$ depending on $\ell$ and for every $1\leq j\leq J$ an algebraic variety $\mathcal{V}_j$ defined over $\mathbb{Q}$ and a lattice $\Lambda_j$ for which equations can be written down explicitly and (Laurent-)polynomials $f_j,h_j\in\mathbb{Q}[\mathcal{V}_j][z_1^{\pm 1},\ldots,z_p^{\pm 1}],g_j\in\mathbb{Q}[\mathcal{V}_j][z]$ with coefficients in the coordinate ring of the variety such that the following holds:
	\begin{compactitem}
		\item[a)] $g_j\circ h_j=f_j$ is a (Laurent-)polynomial with $\ell$ non-constant terms with coefficients in the coordinate ring;
		\item[b)] for every point $P\in\mathcal{V}_j(\mathbb{C})$ and $(u_1,\ldots,u_p)\in\Lambda_j$ one gets a decomposition $f_j(P,x^{u_1},$ $\ldots,x^{u_p})=g_j(P,h_j(P,x^{u_1},\ldots,x^{u_p}))$;
		\item[c)] conversely, for every polynomial $f\in\mathbb{C}[x]$ with $\ell$ non-constant terms and every non-trivial decomposition $f(x)=g\circ h$ with $h(x)$ not of the shape $ax^m+b,m\in\mathbb{N},a,b\in\mathbb{C}$ there is a $j$, a point $P\in\mathcal{V}_j(\mathbb{C})$ and $(u_1,\ldots,u_p)\in\Lambda_j$ such that $f(x)=$ $f_j(P,x^{u_1},\ldots,x^{u_p}),g(x)=g_j(P,x),h(x)=h_j(P,x^{u_1},\ldots,x^{u_p})$.
	\end{compactitem}
	This result is based on an intermediate result \cite{Z3} that the outer decomposition factor has degree bounded explicitly in terms of $l$ unless the inner decomposition factor is cyclic.

	Another instance of this approach is given by {$\mathcal{I}=\{G_n(x); n\in\mathbb{N}\}$}, where $G_n(x)$ are elements of a linear recurrence sequence $(G_n)_{n=0}^\infty$ of polynomials in $\mathbb{C}[x]$.
	To fix terms we shall assume that the recurrence is given by $G_{n+d}(x)=A_{d-1}(x)G_{n+d-1}(x)+\cdots+A_0(x)G_n(x),$ with $A_0,\ldots,A_{d-1}\in\mathbb{C}\left[x\right]$ and initial terms $G_0,\ldots,G_{d-1}\in\mathbb{C}[x]$. Denote by $\alpha_1,\ldots,\alpha_t$ the distinct characteristic roots of the sequence, that is the characteristic polynomial $\mathcal{G}\in\mathbb{C}(x)[T]$ splits as
	$
	\mathcal{G}(T)=T^d-A_{d-1}T^{d-1}-\cdots - A_0=(T-\alpha_1)^{k_1}(T-\alpha_2)^{k_2}\cdots(T-\alpha_t)^{k_t},
	$
	where $k_1,\ldots,k_t\in\mathbb{N}$. Then $G_n(x)$ admits a representation of the form
	$
	G_n(x)=a_1\alpha_1^n+a_2\alpha_2^n+\cdots+a_t\alpha_t^n.
	$
	We say that the recurrence is {minimal} if $(G_n)_{n=0}^{\infty}$ does not satisfy a recurrence relation with smaller $d$ and coefficients in $\mathbb{C}[x]$.
	We say that the recurrence is {non-degenerate} if $\alpha_i/\alpha_j\not\in \mathbb{C}^*$ for all $i\neq j$.
	We say that the recurrence is {simple} if $k_1=\cdots=k_t=1$; in this case the $a_i$'s lie in $\mathbb{C}(x,\alpha_1,\ldots,\alpha_t)$.
	We say that the recurrence is a {polynomial power sum} if $a_1,\ldots,a_d\in\mathbb{C}$ and $\alpha_1,\ldots,\alpha_d\in\mathbb{C}[x]$.
	We say that a polynomial power sum satisfies the {dominant root condition} if $\deg(\alpha_1)>\deg(\alpha_i)$ for $i>1$.
	As an important starting point and motivation we mention that for a given sequence $(G_n)_{n=0}^\infty$ the decompositions of the form $G_n(x)=G_m\circ h$ for a fixed polynomial $h\in\mathbb{C}[x], \deg h\geq 2$ were considered by Peth\H{o}, Tichy and the first author in a series of papers \cite{FPT1,F,FPT2,FPT3}. It was again Zannier \cite{Z2} who proved in general that this equation has only finitely many solutions $(n,m),n\neq m$, unless $h$ is cyclic or dihedral; in this case there are infinitely many solutions coming from a generic equation.
	Moreover, one has to take the following {trivial situations} into account:
	If $G_m(x)\in\mathbb{C}[h(x)]$ for every $m\in\mathbb{N}$, then it is not possible to bound the degree of $g$ independently of $n$ assuming $G_n=g\circ h$.
	If $G_n(x)=g(H_n(x))$ with $g\in\mathbb{C}[x],\deg g=m$ and $(H_n)_{n=0}^\infty$ is another linear recurrence sequence in $\mathbb{C}[x]$, then obviously we again have a sought decomposition for every $n\in\mathbb{N}$. Consider as a nice example the Fibonacci polynomials $F_n$ defined by $F_0(x)=0, F_1(x)=1,$ $F_{n+2}(x)=x F_{n+1}(x)+F_{n}(x)$. It is easy to see that for all odd $n\ge 3$, $F_n$ is an even polynomial of degree $n-1$, and hence if $n\ge5$ is odd, $F_n(x)$ can be written as $F_n(x)=g\circ h$, where $h(x)=x^2$ and $\deg g=(n-1)/2$. Observe that $h$ is cyclic and  that the degree of $g$ cannot be bounded independently of $n$ assuming $F_n(x)=g(h(x))$ and $\deg h>1$.
	Also, for Chebyshev polynomials $T_n$ it is well-known that $T_{mn}(x)=T_m\circ T_n$ for any $m, n\in \mathbb{N}$. Observe that $h$ is dihedral and, since $\deg T_n=n$, one cannot bound $\deg g$ independently of $n$ assuming $T_n(x)=g(h(x))$ and $\deg h>1$.
	
	The main result of the first author proved together with Karolus and Kreso in \cite{FKK} is the following:
	Let $(G_n)_{n=0}^\infty$ be a minimal non-degenerate simple second order linear recurrence sequence. Assume that $G_n$ is decomposable for some $n\in \mathbb{N}$  and write $G_n(x)=g(h(x))$, where $h$ is indecomposable.
	Since $\deg h\geq 2$, there exists a root $y\neq x$ in its splitting field over $\mathbb{C}(h(x))$. Clearly, $h(x)=h(y)$.
	We have $G_n(x)=\pi_1\alpha_1^n+\pi_2\alpha_2^n$.
	Conjugating (in some fixed algebraic closure of $\mathbb{C}(x)$ containing $\alpha_1,\alpha_2$) over $\mathbb{C}(h(x))$ via $x\mapsto y$, we get a sequence $(G_n(y))_{n=0}^\infty$ with $G_n(y)\in\mathbb{C}[y]$, which satisfies the same minimal non-degenerate simple recurrence relation as $(G_n(x))_{n=0}^\infty$ with $x$ replaced by $y$. We conclude that $G_n(y)=\rho_1\beta_1^n+\rho_2\beta_2^n$. Since $h(x)=h(y)$, we get $G_n(x)=G_n(y)$, that is
	\begin{equation}
		\label{p2-eq:sumstar}
		\pi_1\alpha_1^n + \pi_2\alpha_2^n=\rho_1\beta_1^n+\rho_2\beta_2^n.\tag{$\star$}
	\end{equation}
	Then there is a positive real constant $C=C(\{A_i, G_i; i=1,2\})$ with the following property: If for some $n$ we have $G_n(x)=g(h(x))$, where $h$ is indecomposable and neither dihedral nor cyclic, and if \eqref{p2-eq:sumstar} has no proper vanishing subsum, then it holds that $\deg g\leq C$.
	We remark that if $h$ is not cyclic, then equation \eqref{p2-eq:sumstar} has a proper vanishing subsum if and only if
	$
	\pi_1\pi_2A_0(x)^n\in \mathbb{C}(h(x)).
	$
	In particular, the existence of a proper vanishing subsum does not depend on the choice of the conjugate $y$ of $x$ over $\mathbb{C}(h(x))$.
	However, \eqref{p2-eq:sumstar} clearly depends on $n$ and $h$ for which $G_n(x)=g(h(x))$ which are not known a priori.
	Note that if $h$ is not cyclic and $A_0(x)=a_0\in \mathbb{C}$, $\pi_1\pi_2=\pi\in \mathbb{C}$,  then there  exists a vanishing subsum of \eqref{p2-eq:sumstar} and one cannot apply the theorem in question; for example, this is the case for Chebyshev polynomials $T_n$.
	It is possible to give sufficient conditions in which \eqref{p2-eq:sumstar} has no proper vanishing subsum. We do not give further details here.

	Furthermore in \cite{FK} the first author and Karolus proved the following:
	Let $(G_n)_{n=0}^\infty$ be a non-degenerate polynomial power sum which satisfies the dominant root condition.
	Moreover, let $m\geq 2$ be an integer. Write $m_0$ for the least integer such that $\alpha_1^{m_0/m}\in \mathbb{C}[x]$. Then there is an effectively computable positive constant $C$ such that the following holds: Assume that for some $n\in\mathbb{N}$ with $n>C$ we have $G_n(x)=g\circ h$ with $\deg g=m,\deg h>1$. Then there are $c_1,\ldots,c_l\in\mathbb{C}$ such that \[h(x)=c_1\gamma_1^\ell+\cdots +c_l\gamma_l^\ell,\] where $m_0\ell=n$ and $l\in\mathbb{N}$ is bounded explicitly in terms of $m,d$ and $\deg(\alpha_1)+\cdots+\deg(\alpha_d)$ and $\gamma_1,\ldots,\gamma_l\in\mathbb{C}(x)$ can be given explicitly in terms of $\alpha_1,\ldots,\alpha_d$, both independently of $n$.
	Furthermore, it follows that there is an explicitly computable positive constant $C$, and a subvariety $\mathcal{V}$ of $\mathbb{A}^{l+m+1}\times\mathbb{G}_\m^{t}$ with $t,l$ bounded explicitly in terms of $m,d$ and $\deg(\alpha_1)+\cdots+\deg(\alpha_d)$ for which a system of polynomial-exponential equations in the polynomial variables $c_1,\ldots,c_l,g_0,\ldots,g_m$ and the exponential variable $\ell$ (with coefficients in $\mathbb{Q}$) can be written down explicitly such that the following holds:
	\begin{compactitem}
		\item[a)] Defining $G(x)=g_0x^m+g_1x^{m-1}+\cdots +g_m\in\mathbb{C}[\mathcal{V}][x]$ and $H_\ell=c_1\gamma_1^\ell+c_2\gamma_2^\ell+\cdots +c_l\gamma_l^\ell\in\mathbb{C}[\mathcal{V}][x]$, where $\gamma_1,\ldots,\gamma_l\in\mathbb{C}(x)$ can be given explicitly in terms of $\alpha_1,\ldots,\alpha_d$, then $G_{m_0\ell}=G\circ H_\ell$ holds as an equation in $x$ with coefficients in the coordinate ring of $\mathcal{V}$. In particular, for any point $P=(c_1,\ldots,c_l,g_0,\ldots,g_m,\ell)\in\mathcal{V}(\mathbb{C})$ we get a decomposition $G_n(x)=g\circ h$,  $g(x)=G(P,x)\in\mathbb{C}[x]$ and $h(x)=H_l(P,x)\in\mathbb{C}[x]$ (with $n=m_0\ell$).
		\item[b)] Conversely, let $G_n(x)=g\circ h$ be a decomposition of $G_n(x)$ for some $n\in\mathbb{N}$ with $g,h\in\mathbb{C}[x],\deg g=m, \deg h>1$. Then either $n\leq C$ or there exists a point $P=(g_0,\ldots,g_m,c_1,\ldots,c_l,\ell)\in\mathcal{V}(\mathbb{C})$ with $g(x)=G(P,x)$ and $h(x)=H_\ell(P,x)$ and $n=m_0\ell$.
	\end{compactitem}
	A number of interesting special cases follow and are discussed, in particular that the results include a description in finite terms of all $m$-th powers in a linear recurrence sequence of polynomials satisfying the conditions of the theorem.

	Observe that in \cite{FKK} only {binary} recurrences are covered and that in \cite{FK} the order is not restricted but instead only {polynomials power sums satisfying the dominant root condition} are handled.

	In this paper we revisit the situation when the outer polynomial is fixed to be $ g(x) = x^m $. We first quickly review the situation for lacunary polynomials. Here it is natural to consider a non-constant complex polynomial with constant term equal to $1$ and with $k$ additional non-constant terms. The results \cite{S} and \cite{Z1} immediately imply that then $m\leq k$. For $k\leq 3$ a precise classification of all solutions can be found in \cite{CZ} (see Lemma 2.1). The analogous result for $k=4$ was given in the recent PhD thesis \cite{M}. Now we turn back to polynomial power sums. Assuming the dominant root condition we will prove that the second case which states that $ h $ is of a special form cannot occur in this setting. For the binary case we will be able to prove a stronger result than for the general one of order greater than two. We are going to give a counterexample which shows that the stronger result is in general not true for the case of an arbitrary order $ d $ of the linear recurrence sequence.

	\section{Results}
	
	During the whole paper we are implicitly assuming that the polynomial $ h(x) $ has degree $ \deg h \geq 2 $. Let us now first state our two results that we are going to prove in the next section: We start with the situation of binary recurrences.
	
	\begin{mythm}
		\label{p2-thm:binarycase}
		Let $ (G_n(x))_{n=0}^{\infty} $ be a non-degenerate simple linear recurrence sequence of order $ d = 2 $ with power sum representation $ G_n(x) = a_1 \alpha_1^n + a_2 \alpha_2^n $ such that $ \alpha_1, \alpha_2 \in \CC[x] $ are polynomials and $ a_1, a_2 \in \CC(x) $ satisfy $ \frac{a_2}{a_1} \in \CC $. Assume furthermore that $ \deg \alpha_1 > \deg \alpha_2 $.
		Then there exists a constant $ C $, which depends only on $ \alpha_1, \alpha_2, a_1, a_2 $, such that for all $ n > C $ there is no integer $ m \geq 2 $ and no polynomial $ h(x) \in \CC[x] $ with the property $ G_n(x) = (h(x))^m $. In particular this implies that for $ m $ large enough there is no index $ n $ and no polynomial $ h(x) \in \CC[x] $ such that $ G_n(x) = (h(x))^m $.
	\end{mythm}
	
	In the case of recurrences of arbitrary large order we prove a slightly weaker result (essentially, the final conclusion in the previous theorem).

	\begin{mythm}
		\label{p2-thm:generalcase}
		Let $ (G_n(x))_{n=0}^{\infty} $ be a non-degenerate simple linear recurrence sequence of order $ d \geq 3 $ with power sum representation $ G_n(x) = a_1 \alpha_1^n + \cdots +  a_d \alpha_d^n $ such that $ \alpha_1, \ldots, \alpha_d \in \CC[x] $ are polynomials and $ a_1, \ldots, a_d \in \CC $ are constant. Assume furthermore that $ \deg \alpha_1 > \deg \alpha_2 > \deg \alpha_3 \geq \deg \alpha_4 \geq \cdots \geq \deg \alpha_d $.
		Then for $ m $ large enough there is no index $ n $ and no polynomial $ h(x) \in \CC[x] $ such that $ G_n(x) = (h(x))^m $.
	\end{mythm}
	
	In the case of a non-degenerate simple linear recurrence sequence of order $ d \geq 3 $ we cannot give in general a bound $ C $ for the index $ n $ such that all elements of the sequence $ (G_n(x))_{n=0}^{\infty} $ with index $ n > C $ are no proper powers.
	Consider for instance the third order sequence given by $ G_n(x) = (x^n+1)^2 = (x^2)^n + 2x^n + 1^n $ which has the property that each element is at least a square.
	We can easily modify this example to generate counterexamples for any fixed parameter $ m $ if we consider $ G_n(x) = (x^n+1)^m $.
	
	For the proof we mainly follow the proof of \cite{FK}.
	Therefore, we start from $G_n(x)=h(x)^m$. Thus $h(x)=\zeta G_n(x)^{1/m}$ (as formal power series).
	Then one uses the multinomial series to expand $G_n(x)^{1/m}$; in order to justify this multiple expansion, the dominant root condition on the degrees of the characteristic roots is needed. Afterwards a function field variant of the Schmidt subspace theorem (due to Zannier, see Proposition \ref{p2-prop:functionfieldsubspace} below) is used, to find that either $n$ is bounded or $h$ can be expressed in the form $c_0t_0+c_1t_1+\cdots +c_{L-1}t_{L-1}$, where $c_i\in\CC$ and $t_i$ come from a finite set. We have to show that the latter case is impossible. Plugging the expression for $h(x)$ into $G_n(x)=h(x)^m$ and comparing degrees and leading coefficients gives the result.

	The proof of the two theorems are quite similar, the difference involve some subtleties that we try to work out. We remark that the method of proof exactly requires $(G_n(x))_{n=0}^\infty$ to be a polynomial power sum with dominant root condition.

	\section{Preliminaries}

	In the sequel we will need the following notations: For $ c \in \CC $ and $ f(x) \in \CC(x) $ where $ \CC(x) $ is the rational function field over $ \CC $ denote by $ \nu_c(f) $ the unique integer such that $ f(x) = (x-c)^{\nu_c(f)} p(x) / q(x) $ with $ p(x),q(x) \in \CC[x] $ such that $ p(c)q(c) \neq 0 $. Further denote by $ \nu_{\infty}(f) = \deg q - \deg p $ if $ f(x) = p(x) / q(x) $.
	These functions $ \nu $ are up to equivalence all valuations on $ \CC(x) $. If $ \nu_c(f) > 0 $, then $ c $ is called a zero of $ f $, and if $ \nu_c(f) < 0 $, then $ c $ is called a pole of $ f $.
	For a finite extension $ F $ of $ \CC(x) $ each valuation on $ \CC(x) $ can be extended to no more than $ [F : \CC(x)] $ valuations on $ F $. This again gives all valuations on $ F $.
	Both, in $ \CC(x) $ as well as in $ F $ the sum-formula
	\begin{equation*}
		\sum_{\nu} \nu(f) = 0
	\end{equation*}
	holds, where $ \sum_{\nu} $ means that the sum is taken over all valuations on the considered function field.
	Each valuation on a function field corresponds to a place and vice versa. The set of all places of the function field $ F $ is denoted by $ \PP_F $.
	If $ F' $ is a finite extension of $ F $, then we say that $ P' \in \PP_{F'} $ lies over $ P \in \PP_F $ if $ P \subseteq P' $ and denote this fact by $ P' \mid P $.
	In this case there exists an integer $ e(P' \mid P) $, the so-called ramification index of $ P' $ over $ P $, such that for all $ x \in F $ the equality $ \nu_{P'}(x) = e(P' \mid P) \cdot \nu_P(x) $ holds.
	
	To prepare the proofs of our two theorems we present subsequently three auxiliary results that are used in \cite{FK} as well. The first one also can be found in \cite{St}:
	
	\begin{myprop}
		\label{p2-prop:kummerext}
		Let $ F/\CC $ be a function field in one variable. Suppose that $ u \in F $ satisfies $ u \neq w^d $ for all $ w \in F $ and $ d \mid n $, $ d > 1 $. Let $ F' = F(z) $ with $ z^n = u $. Then $ F' $ is said to be a Kummer extension of $ F $ and we have:
		\begin{enumerate}[a)]
			\item The polynomial $ \varphi(T) = T^n - u $ is the minimal polynomial of $ z $ over $ F $ (in particular, it is irreducible over $ F $). The extension $ F'/F $ is Galois of degree $ n $; its Galois group is cyclic and all automorphisms of $ F'/F $ are given by $ \sigma(z) = \zeta z $, where $ \zeta \in \CC $ is an $ n $-th root of unity.
			\item Let $ P \in \PP_F $ and $ P' \in \PP_{F'} $ be an extension of $ P $. Let $ r_P := \gcd \setb{n, \nu_P(u)} $. Then $ e(P'|P) = n/r_P $.
			\item Denote by $ \gfr $ (resp. $ \gfr' $) the genus of $ F/\CC $ (resp. $ F'/\CC $). Then
			\begin{equation*}
				\gfr' = 1 + n(\gfr-1) + \frac{1}{2} \sum_{P \in \PP_F} (n-r_P) \deg P.
			\end{equation*}
		\end{enumerate}
	\end{myprop}
	
	The following proposition can be seen as a function field analogue of the Schmidt subspace theorem. It will play an important role in our proofs. The reader can find a proof for it in \cite{Z1}:
	
	\begin{myprop}[Zannier]
		\label{p2-prop:functionfieldsubspace}
		Let $ F/\CC $ be a function field in one variable, of genus $ \gfr $, let $ \varphi_1, \ldots, \varphi_n \in F $ be linearly independent over $ \CC $ and let $ r \in \set{0,1, \ldots, n} $. Let $ S $ be a finite set of places of $ F $ containing all the poles of $ \varphi_1, \ldots, \varphi_n $ and all the zeros of $ \varphi_1, \ldots, \varphi_r $. Put $ \sigma = \sum_{i=1}^{n} \varphi_i $. Then
		\begin{equation*}
			\sum_{\nu \in S} \left( \nu(\sigma) - \min_{i=1, \ldots, n} \nu(\varphi_i) \right) \leq \bino{n}{2} (\abs{S} + 2\gfr - 2) + \sum_{i=r+1}^{n} \deg (\varphi_i).
		\end{equation*}
	\end{myprop}
	
	In the next section we will take use of height functions in function fields. Let us therefore define the height of an element $ f \in F^* $ by
	\begin{equation*}
		\Hc(f) := - \sum_{\nu} \min \setb{0, \nu(f)} = \sum_{\nu} \max \setb{0, \nu(f)}
	\end{equation*}
	where again the sum is taken over all valuations on the function field $ F / \CC $. Additionally we define $ \Hc(0) = \infty $.
	These height function satisfies some basic properties that are listed in the following lemma which is proven in \cite{FKK}:
	
	\begin{mylemma}
		\label{p2-lemma:heightproperties}
		Denote as above by $ \Hc $ the projective height on $ F/\CC $. Then for $ f,g \in F^* $ the following properties hold:
		\begin{enumerate}[a)]
			\item $ \Hc(f) \geq 0 $ and $ \Hc(f) = \Hc(1/f) $,
			\item $ \Hc(f) - \Hc(g) \leq \Hc(f+g) \leq \Hc(f) + \Hc(g) $,
			\item $ \Hc(f) - \Hc(g) \leq \Hc(fg) \leq \Hc(f) + \Hc(g) $,
			\item $ \Hc(f^n) = \abs{n} \cdot \Hc(f) $,
			\item $ \Hc(f) = 0 \iff f \in \CC^* $,
			\item $ \Hc(A(f)) = \deg A \cdot \Hc(f) $ for any $ A \in \CC[T] \setminus \set{0} $.
		\end{enumerate}
	\end{mylemma}
	
	\section{Proofs}
	
	We are now ready to prove our two theorems. At this position we remark that our proofs are very similar to the proof of Theorem 1 in \cite{FK} where the same procedure is used.
	
	\begin{proof}[Proof of Theorem \ref{p2-thm:binarycase}]
		Assume that there exists an index $ n $, an integer $ m \geq 2 $ and a polynomial $ h(x) $ such that $ G_n(x) = (h(x))^m $.
		Thus we have $ h(x) = \zeta (G_n(x))^{1/m} $ for an $ m $-th root of unity $ \zeta $.
		Using the power sum representation of $ G_n(x) $ as well as the binomial series expansion we get
		\begin{align}
			h(x) &= \zeta (G_n(x))^{1/m} = \zeta (a_1 \alpha_1^n + a_2 \alpha_2^n)^{1/m} \nonumber \\
			&= \zeta a_1^{1/m} \alpha_1^{n/m} \left( 1 + \frac{a_2}{a_1} \left( \frac{\alpha_2}{\alpha_1} \right)^n \right)^{1/m} \nonumber \\
			&= \zeta a_1^{1/m} \alpha_1^{n/m} \sum_{h_2=0}^{\infty} \bino{1/m}{h_2} \left( \frac{a_2}{a_1} \right)^{h_2} \left( \frac{\alpha_2}{\alpha_1} \right)^{nh_2} \nonumber \\
			&\label{p2-eq:sumforh}= \sum_{h_2=0}^{\infty} t_{h_2}(x)
		\end{align}
		with the definition
		\begin{align*}
			t_{h_2}(x) :&= \zeta a_1^{1/m} \alpha_1^{n/m} \bino{1/m}{h_2} \left( \frac{a_2}{a_1} \right)^{h_2} \left( \frac{\alpha_2}{\alpha_1} \right)^{nh_2} \\
			&= b_{h_2} a_1^{1/m} \alpha_1^{n/m} \left( \frac{\alpha_2}{\alpha_1} \right)^{nh_2}.
		\end{align*}
		Since we have required in the theorem that $ \frac{a_2}{a_1} \in \CC $, it holds that
		\begin{equation*}
			b_{h_2} := \zeta \bino{1/m}{h_2} \left( \frac{a_2}{a_1} \right)^{h_2} \in \CC.
		\end{equation*}
		
		Let now $ F = \CC(x, \alpha_1(x)^{1/m}) $ and $ m_0 $ be the smallest positive integer such that $ \alpha_1(x)^{m_0/m} \in \CC(x) $.
		Applying Proposition \ref{p2-prop:kummerext} we get that $ F $ is a Kummer extension of $ \CC(x) $ and that $ T^{m_0} - \alpha_1(x)^{m_0/m} $ is the minimal polynomial of $ \alpha_1(x)^{1/m} $ over $ \CC(x) $.
		Moreover, we get that only places in $ F $ above $ \infty $ and the roots of $ \alpha_1 $ as a polynomial in $ \CC(x) $ can ramify.
		Since our field of constants is $ \CC $ and therefore algebraically closed, we have $ \deg P = 1 $ for all places $ P $. Combined with $ \gfr_{\CC(x)} = 0 $ the genus formula of Proposition \ref{p2-prop:kummerext} yields
		\begin{align*}
			2\gfr_F - 2 &= 2 m_0 (\gfr_{\CC(x)} - 1) + \sum_{P \in \PP_{\CC(x)}} (m_0-r_P) \deg P \\
			&\leq -2 m_0 + \sum_{P \in \PP_{\CC(x)} : m_0 > r_P} m_0 \\
			&\leq -2 m_0 + m_0 (1 + \deg \alpha_1) = m_0 (\deg \alpha_1 - 1).
		\end{align*}
		
		Moreover, let $ F' = F(a_1^{1/m}) $ and $ m_1 $ be the smallest positive integer such that $ a_1^{m_1/m} \in F $. Again the application of Proposition \ref{p2-prop:kummerext} yields that $ F' $ is a Kummer extension of $ F $.
		Furthermore, we get that only places in $ F' $ above $ \infty $ and the zeros and poles of $ a_1 $ as an element of $ \CC(x) $ can ramify.
		Since our field of constants is $ \CC $ and therefore algebraically closed, we have $ \deg P = 1 $ for all places $ P $. Combined with the bound on $ \gfr_F $ the genus formula of Proposition \ref{p2-prop:kummerext} yields
		\begin{align*}
			2\gfr_{F'} - 2 &= m_1 (2\gfr_F - 2) + \sum_{P \in \PP_F} (m_1-r_P) \deg P \\
			&\leq m_1 m_0 (\deg \alpha_1 - 1) + \sum_{P \in \PP_F : m_1 > r_P} m_1 \\
			&\leq m_1 m_0 (\deg \alpha_1 - 1) + m_1 m_0 (1 + 2\Hc(a_1)) \\
			&= m_1 m_0 (\deg \alpha_1 + 2\Hc(a_1)).
		\end{align*}
		
		The next step is to estimate the valuation of the $ t_{h_2} $ corresponding to the infinite place of $ \CC(x) $. We get the following lower bound:
		\begin{align*}
			\nu_{\infty}(t_{h_2}(x)) &= \nu_{\infty}\left( a_1^{1/m} \alpha_1^{n/m} \left( \frac{\alpha_2}{\alpha_1} \right)^{nh_2} \right) \\
			&= \frac{1}{m} \nu_{\infty}(a_1) + n \left( \frac{1}{m} \nu_{\infty}(\alpha_1) + h_2 (\nu_{\infty}(\alpha_2) - \nu_{\infty}(\alpha_1)) \right) \\
			&= \frac{1}{m} \nu_{\infty}(a_1) + n \left( \frac{-\deg \alpha_1}{m} + h_2 (\deg \alpha_1 - \deg \alpha_2) \right) \\
			&\geq \frac{1}{m} \nu_{\infty}(a_1) + n \left( h_2 - \frac{\deg \alpha_1}{m} \right).
		\end{align*}
		Let $ J \in \NN $ be arbitrary. Therefore for $ h_2 \geq J + \frac{\deg \alpha_1}{m} $ we have $ \nu_{\infty}(t_{h_2}(x)) \geq \frac{1}{m} \nu_{\infty}(a_1) + nJ $.
		This allows us to split the above sum representation \eqref{p2-eq:sumforh} for $ h(x) $ in the following way:
		\begin{equation*}
			h(x) = t_0(x) + t_1(x) + \cdots + t_{L-1}(x) + \sum_{\nu_{\infty}(t_{h_2}(x)) \geq \frac{1}{m} \nu_{\infty}(a_1) + nJ} t_{h_2}(x)
		\end{equation*}
		with $ L-1 < J + \frac{\deg \alpha_1}{m} $.
		
		We now distinguish between two cases which will be handled in completely different ways. First we assume that $ \set{h(x), t_0(x), \ldots, t_{L-1}(x)} $ is linearly independent over $ \CC $. Later we will consider the case that $ \set{h(x), t_0(x), \ldots, t_{L-1}(x)} $ is linearly dependent over $ \CC $.
		
		So let us now assume that $ \set{h(x), t_0(x), \ldots, t_{L-1}(x)} $ is linearly independent over $ \CC $. We aim to apply Proposition \ref{p2-prop:functionfieldsubspace}.
		To do so let us fix a finite set $ S $ of places of $ F' $ which contains all zeros and poles of $ t_0(x), \ldots, t_{L-1}(x) $ as well as all poles of $ h(x) $. Therefore $ S $ can be chosen in a way such that it contains at most the places above $ \infty $, the zeros of $ \alpha_1 $ and $ \alpha_2 $ and the zeros and poles of $ a_1 $. This gives an upper bound on the number of elements in $ S $:
		\begin{equation*}
			\abs{S} \leq m_1 m_0 (1 + \deg \alpha_1 + \deg \alpha_2 + 2\Hc(a_1)).
		\end{equation*}
		
		Further we write $ \varphi_0 = -t_0(x), \ldots, \varphi_{L-1} = -t_{L-1}(x) $ and $ \varphi_L = h(x) $. We also define $ \sigma = \sum_{i=0}^{L} \varphi_i = \sum_{\nu_{\infty}(t_{h_2}(x)) \geq \frac{1}{m} \nu_{\infty}(a_1) + nJ} t_{h_2}(x) $.
		Since $ \deg (h(x)) = [F' : \CC(h(x))] = [F' : F] \cdot [F : \CC(h(x))] = m_1 \Hc(h(x)) = m_1 \deg h \cdot \Hc(x) = m_1 \deg h \cdot [F : \CC(x)] = m_1 m_0 \deg h $ Proposition \ref{p2-prop:functionfieldsubspace} implies
		\begin{align*}
			\sum_{\nu \in S} &\left( \nu(\sigma) - \min_{i=0, \ldots, L} \nu(\varphi_i) \right) \leq \\
			&\leq \bino{L+1}{2} (\abs{S} + 2\gfr_{F'} - 2) + \deg (h(x)) \\
			&\leq \frac{1}{2} L (L+1) m_1 m_0 (1 + 2 \deg \alpha_1 + \deg \alpha_2 + 4 \Hc(a_1)) + m_1 m_0 \deg h \\
			&\leq L (L+1) m_1 m_0 (1 + \deg \alpha_1 + \deg \alpha_2 + 2\Hc(a_1)) + m_1 m_0 \deg h.
		\end{align*}
		On the other hand we have $ \nu(\sigma) \geq \min_{i=0, \ldots, L} \nu(\varphi_i) $ for every valuation $ \nu $ and thus the lower bound
		\begin{align*}
			\sum_{\nu \in S} \left( \nu(\sigma) - \min_{i=0, \ldots, L} \nu(\varphi_i) \right)
			&\geq \sum_{P \mid \infty} \left( \nu_P(\sigma) - \min_{i=0, \ldots, L} \nu_P(\varphi_i) \right) \\
			&\geq \sum_{P \mid \infty} \left( \nu_P(\sigma) - \nu_P(h(x)) \right) \\
			&= \sum_{P \mid \infty} \nu_P(\sigma) - \sum_{P \mid \infty} e(P \mid \infty) \cdot \nu_{\infty}(h(x)) \\
			&= \sum_{P \mid \infty} \nu_P(\sigma) - m_1 m_0 \nu_{\infty}(h(x)) \\
			&= \sum_{P \mid \infty} e(P \mid \infty) \cdot \nu_{\infty}(\sigma) + m_1 m_0 \deg h \\
			&= m_1 m_0 \nu_{\infty}(\sigma) + m_1 m_0 \deg h \\
			&\geq \frac{m_1 m_0}{m} \nu_{\infty}(a_1) + m_1 m_0 nJ + m_1 m_0 \deg h.
		\end{align*}
		Let us now compare the upper and lower bounds. Since $ m_1 m_0 \deg h $ appears on both sides, we can subtract it and get
		\begin{equation*}
			\frac{m_1 m_0}{m} \nu_{\infty}(a_1) + m_1 m_0 nJ \leq L (L+1) m_1 m_0 (1 + \deg \alpha_1 + \deg \alpha_2 + 2\Hc(a_1)).
		\end{equation*}
		Dividing by $ m_1 m_0 $ and isolating the term containing $ n $ yields
		\begin{equation*}
			nJ \leq  L (L+1) (1 + \deg \alpha_1 + \deg \alpha_2 + 2\Hc(a_1)) + \abs{\nu_{\infty}(a_1)}.
		\end{equation*}
		Since $ J \in \NN $ was arbitrary we can now choose $ J = 1 $. Remember that $ L-1 < J + \frac{\deg \alpha_1}{m} $ and therefore $ L \leq 1 + J + \deg \alpha_1 = 2 + \deg \alpha_1 $. Hence
		\begin{equation}
			\label{p2-eq:boundforn}
			n \leq (2 + \deg \alpha_1) (3 + \deg \alpha_1) (1 + \deg \alpha_1 + \deg \alpha_2 + 2\Hc(a_1)) + \abs{\nu_{\infty}(a_1)}.
		\end{equation}
		
		After this we consider now the case that $ \set{h(x), t_0(x), \ldots, t_{L-1}(x)} $ is linearly dependent over $ \CC $.
		We can assume that $ \set{t_0(x), \ldots, t_{L-1}(x)} $ is linearly independent, since otherwise we are able to group them together and the first case is still working if the $ t_{h_2}(x) $ have constant coefficients.
		This implies that in a relation of linear dependence $ h(x) $ has a nonzero coefficient. Thus there exist complex numbers $ c_i \in \CC $ such that
		\begin{equation}
			\label{p2-eq:binarylasum}
			h(x) = \sum_{i=0}^{L-1} c_i t_i(x).
		\end{equation}
		
		What we are doing subsequently is a reverse induction. We will show $ c_{L-1} = 0 $, then $ c_{L-2} = 0 $ and so on until only $ c_0 $ remains.
		During the following calculations we will use the abbreviations $ \beta_1 := \alpha_1^n $ and $ \beta_2 := \alpha_2^n $.
		Furthermore let $ d_i = b_i c_i $.
		We start with equation \eqref{p2-eq:binarylasum} and get
		\begin{align*}
			h(x) &= c_0 t_0(x) + \cdots + c_{L-1} t_{L-1}(x) \\
			&= d_0 a_1^{1/m} \alpha_1^{n/m} + d_1 a_1^{1/m} \alpha_1^{n/m} \left( \frac{\alpha_2}{\alpha_1} \right)^{n} + \cdots + d_{L-1} a_1^{1/m} \alpha_1^{n/m} \left( \frac{\alpha_2}{\alpha_1} \right)^{n(L-1)} \\
			&= d_0 a_1^{1/m} \beta_1^{1/m} + d_1 a_1^{1/m} \beta_1^{1/m} \frac{\beta_2}{\beta_1} + \cdots + d_{L-1} a_1^{1/m} \beta_1^{1/m} \left( \frac{\beta_2}{\beta_1} \right)^{L-1} \\
			&= a_1^{1/m} \beta_1^{1/m} \left( d_0 + d_1 \frac{\beta_2}{\beta_1} + \cdots + d_{L-1} \left( \frac{\beta_2}{\beta_1} \right)^{L-1} \right)
		\end{align*}
		as well as
		\begin{equation*}
			a_1 \beta_1 + a_2 \beta_2 = (h(x))^m = a_1 \beta_1 \left( d_0 + d_1 \frac{\beta_2}{\beta_1} + \cdots + d_{L-1} \left( \frac{\beta_2}{\beta_1} \right)^{L-1} \right)^m.
		\end{equation*}
		Multiplying with $ \beta_1^{m(L-1)} $ and dividing by $ a_1 $ yields
		\begin{align*}
			\beta_1^{1+m(L-1)} + \frac{a_2}{a_1} \beta_1^{m(L-1)} \beta_2 &= \beta_1 \left( d_0 \beta_1^{L-1} + d_1 \beta_1^{L-2} \beta_2 + \cdots + d_{L-1} \beta_2^{L-1} \right)^m \\
			&= d_0^m \beta_1^{1+m(L-1)} + m d_0^{m-1} d_1 \beta_1^{m(L-1)} \beta_2 \\
			&\hspace{0.5cm}+ \left( \bino{m}{2} d_0^{m-2} d_1^2 + m d_0^{m-1} d_2 \right) \beta_1^{m(L-1)-1} \beta_2^2 \\
			&\hspace{0.5cm}+ \cdots + d_{L-1}^m \beta_1 \beta_2^{m(L-1)}.
		\end{align*}
		Now we take a closer look at the coefficients of the monomials $ \beta_1^i \beta_2^j $ in the above equation.
		We can rewrite the last equation as
		\begin{align}
			(1-d_0^m) \beta_1^{1+m(L-1)} &= \left( m d_0^{m-1} d_1 - \frac{a_2}{a_1} \right) \beta_1^{m(L-1)} \beta_2 \nonumber \\
			&\label{p2-eq:comparedegstep1}\hspace{0.5cm}+ \left( \bino{m}{2} d_0^{m-2} d_1^2 + m d_0^{m-1} d_2 \right) \beta_1^{m(L-1)-1} \beta_2^2 \\
			&\hspace{0.5cm}+ \cdots + d_{L-1}^m \beta_1 \beta_2^{m(L-1)}. \nonumber
		\end{align}
		The left hand side of this equation is either zero or a polynomial of degree equal to $ (1+m(L-1)) \deg \beta_1 $, whereas the polynomial on the right hand side can have at most degree $ m(L-1) \deg \beta_1 + \deg \beta_2 $. Since they are equal both sides must be zero.
		We get $ 1-d_0^m = 0 $ and rewrite the expression on the right side of \eqref{p2-eq:comparedegstep1} as
		\begin{align*}
			\left( \frac{a_2}{a_1} - m d_0^{m-1} d_1 \right) \beta_1^{m(L-1)} \beta_2 &= \left( \bino{m}{2} d_0^{m-2} d_1^2 + m d_0^{m-1} d_2 \right) \beta_1^{m(L-1)-1} \beta_2^2 \\
			&\hspace{0.5cm}+ \cdots + d_{L-1}^m \beta_1 \beta_2^{m(L-1)}.
		\end{align*}
		We apply the same argument as before to get $ \frac{a_2}{a_1} - m d_0^{m-1} d_1 = 0 $. Now we repeat this procedure and end up with
		\begin{align*}
			1-d_0^m &= 0 \\
			\frac{a_2}{a_1} - m d_0^{m-1} d_1 &= 0 \\
			\bino{m}{2} d_0^{m-2} d_1^2 + m d_0^{m-1} d_2 &= 0 \\
			&\vdots \\
			d_{L-1}^m &= 0.
		\end{align*}
		It follows immediately that $ d_{L-1} = 0 $. Hence $ c_{L-1} = 0 $.
		Thus equation \eqref{p2-eq:binarylasum} reduces to
		\begin{equation*}
			h(x) = \sum_{i=0}^{L-2} c_i t_i(x).
		\end{equation*}
		Doing the same calculations again with this new sum or putting $ d_{L-1} = 0 $ and inspect the above calculations in more detail gives now step by step $ c_{L-2} = 0, \ldots, c_1 = 0 $.
		So it holds that $ h(x) = c_0 t_0(x) $.
		Taking the $ m $-th power we get
		\begin{equation*}
			a_1 \alpha_1^n + a_2 \alpha_2^n = d_0^m a_1 \alpha_1^n
		\end{equation*}
		which is equivalent to
		\begin{equation*}
			\frac{a_2}{a_1} \alpha_2^n = (d_0^m-1) \alpha_1^n.
		\end{equation*}
		The left hand side is a polynomial of degree $ n \deg \alpha_2 $, but the right hand side is either zero or of degree $ n \deg \alpha_1 $. This is a contradiction.
		Therefore the linear dependent case cannot occur.
		
		Altogether we must be in the linear independent case and thus have the bound \eqref{p2-eq:boundforn} for the index $ n $. This proves the theorem.
	\end{proof}
	
	The proof of the other theorem is very similar but has some subtle differences. Hence for the readers convenience we write it down in detail.
	
	\begin{proof}[Proof of Theorem \ref{p2-thm:generalcase}]
		Assume that there exists an index $ n $, an integer $ m > \deg \alpha_1 $ and a polynomial $ h(x) $ such that $ G_n(x) = (h(x))^m $.
		Thus again we have $ h(x) = \zeta (G_n(x))^{1/m} $ for an $ m $-th root of unity $ \zeta $.
		Using the power sum representation of $ G_n(x) $ as well as the multinomial series expansion we get
		\begin{align}
			h(x) &= \zeta (G_n(x))^{1/m} = \zeta (a_1 \alpha_1^n + \cdots + a_d \alpha_d^n)^{1/m} \nonumber \\
			&= \zeta a_1^{1/m} \alpha_1^{n/m} \left( 1 + \frac{a_2}{a_1} \left( \frac{\alpha_2}{\alpha_1} \right)^n + \cdots + \frac{a_d}{a_1} \left( \frac{\alpha_d}{\alpha_1} \right)^n \right)^{1/m} \nonumber \\
			&= \zeta a_1^{1/m} \alpha_1^{n/m} \sum_{h_2,\ldots,h_d=0}^{\infty} g_{h_2,\ldots,h_d} \left( \frac{a_2}{a_1} \right)^{h_2} \left( \frac{\alpha_2}{\alpha_1} \right)^{nh_2} \cdots \left( \frac{a_d}{a_1} \right)^{h_d} \left( \frac{\alpha_d}{\alpha_1} \right)^{nh_d} \nonumber \\
			&\label{p2-eq:gsumforh}= \sum_{h_2,\ldots,h_d=0}^{\infty} t_{h_2,\ldots,h_d}(x)
		\end{align}
		with the definition
		\begin{align*}
			t_{h_2,\ldots,h_d}(x) :&= \zeta a_1^{1/m} \alpha_1^{n/m} g_{h_2,\ldots,h_d} \left( \frac{a_2}{a_1} \right)^{h_2} \left( \frac{\alpha_2}{\alpha_1} \right)^{nh_2} \cdots \left( \frac{a_d}{a_1} \right)^{h_d} \left( \frac{\alpha_d}{\alpha_1} \right)^{nh_d} \\
			&= b_{h_2,\ldots,h_d} \alpha_1^{n/m} \left( \frac{\alpha_2}{\alpha_1} \right)^{nh_2} \cdots \left( \frac{\alpha_d}{\alpha_1} \right)^{nh_d}.
		\end{align*}
		Since we have required in the theorem that $ a_1, \ldots, a_d \in \CC $, it holds that
		\begin{equation*}
			b_{h_2,\ldots,h_d} := \zeta a_1^{1/m} g_{h_2,\ldots,h_d} \left( \frac{a_2}{a_1} \right)^{h_2} \cdots \left( \frac{a_d}{a_1} \right)^{h_d} \in \CC.
		\end{equation*}
		
		Let now $ F = \CC(x, \alpha_1(x)^{1/m}) $ and $ m_0 $ be the smallest positive integer such that $ \alpha_1(x)^{m_0/m} \in \CC(x) $.
		Applying Proposition \ref{p2-prop:kummerext} we get that $ F $ is a Kummer extension of $ \CC(x) $ and that $ T^{m_0} - \alpha_1(x)^{m_0/m} $ is the minimal polynomial of $ \alpha_1(x)^{1/m} $ over $ \CC(x) $.
		Moreover we get that only places in $ F $ above $ \infty $ and the roots of $ \alpha_1 $ as a polynomial in $ \CC(x) $ can ramify.
		Since our field of constants is $ \CC $ and therefore algebraically closed, we have $ \deg P = 1 $ for all places $ P $. Combined with $ \gfr_{\CC(x)} = 0 $ the genus formula of Proposition \ref{p2-prop:kummerext} yields
		\begin{align*}
			2\gfr_F - 2 &= 2 m_0 (\gfr_{\CC(x)} - 1) + \sum_{P \in \PP_{\CC(x)}} (m_0-r_P) \deg P \\
			&\leq -2 m_0 + \sum_{P \in \PP_{\CC(x)} : m_0 > r_P} m_0 \\
			&\leq -2 m_0 + m_0 (1 + \deg \alpha_1) = m_0 (\deg \alpha_1 - 1).
		\end{align*}
		
		The next step is to estimate the valuation of the $ t_{h_2,\ldots,h_d} $ corresponding to the infinite place of $ \CC(x) $. We get the following lower bound:
		\begin{align*}
			\nu_{\infty}(t_{h_2,\ldots,h_d}(x)) &= \nu_{\infty}\left( \alpha_1^{n/m} \left( \frac{\alpha_2}{\alpha_1} \right)^{nh_2} \cdots \left( \frac{\alpha_d}{\alpha_1} \right)^{nh_d} \right) \\
			&= n \left( \frac{1}{m} \nu_{\infty}(\alpha_1) + \sum_{j=2}^{d} h_j (\nu_{\infty}(\alpha_j) - \nu_{\infty}(\alpha_1)) \right) \\
			&= n \left( \frac{-\deg \alpha_1}{m} + \sum_{j=2}^{d} h_j (\deg \alpha_1 - \deg \alpha_j) \right) \\
			&\geq n \left( \sum_{j=2}^{d} h_j - \frac{\deg \alpha_1}{m} \right) \geq n \left( \sum_{j=2}^{d} h_j - 1 \right).
		\end{align*}
		Let $ J \in \NN $ be arbitrary. Therefore for $ \sum_{j=2}^{d} h_j \geq J + 1 $ we have the lower bound $ \nu_{\infty}(t_{h_2,\ldots,h_d}(x)) \geq nJ $.
		This allows us to split the above sum representation \eqref{p2-eq:gsumforh} for $ h(x) $ in the following way:
		\begin{equation*}
			h(x) = t_1(x) + t_2(x) + \cdots + t_L(x) + \sum_{\nu_{\infty}(t_{h_2,\ldots,h_d}(x)) \geq nJ} t_{h_2,\ldots,h_d}(x)
		\end{equation*}
		with $ L \leq (J+1)^{d-1} $.
		
		We remark at this point that if $ J = 1 $ we can get a better bound. Let us therefore use the notation $ e_1 := (0,\ldots,0), e_2 := (1,0,\ldots,0), e_3 := (0,1,0,\ldots,0), \ldots, e_d := (0,\ldots,0,1) \in \NN_0^{d-1} $.
		Among the $ t_1(x), \ldots, t_L(x) $ occur in this case at most some elements of $ \set{t_{e_1}(x), \ldots, t_{e_d}(x)} $ and it holds $ L \leq d $.
		Therefore from now on we fix $ J = 1 $.
		
		We now distinguish between two cases which will be handled in completely different ways. First we assume that $ \set{h(x), t_1(x), \ldots, t_L(x)} $ is linearly independent over $ \CC $. Later we will consider the case that $ \set{h(x), t_1(x), \ldots, t_L(x)} $ is linearly dependent over $ \CC $.
		
		So let us now assume that $ \set{h(x), t_1(x), \ldots, t_L(x)} $ is linearly independent over $ \CC $. We aim to apply Proposition \ref{p2-prop:functionfieldsubspace}.
		To do so let us fix a finite set $ S $ of places of $ F $ which contains all zeros and poles of $ t_1(x), \ldots, t_L(x) $ as well as all poles of $ h(x) $. Therefore $ S $ can be chosen in a way such that it contains at most the places above $ \infty $ and the zeros of $ \alpha_1, \ldots, \alpha_d $. This gives an upper bound on the number of elements in $ S $:
		\begin{equation*}
			\abs{S} \leq m_0 \left( 1 + \sum_{j=1}^{d} \deg \alpha_j \right).
		\end{equation*}
		
		Further we write $ \varphi_1 = -t_1(x), \ldots, \varphi_L = -t_L(x) $ and $ \varphi_{L+1} = h(x) $. We also define $ \sigma = \sum_{i=1}^{L+1} \varphi_i = \sum_{\nu_{\infty}(t_{h_2,\ldots,h_d}(x)) \geq nJ} t_{h_2,\ldots,h_d}(x) $.
		Since $ \deg (h(x)) = [F : \CC(h(x))] = \Hc(h(x)) = \deg h \cdot \Hc(x) = \deg h \cdot [F : \CC(x)] = m_0 \deg h $ Proposition \ref{p2-prop:functionfieldsubspace} implies
		\begin{align*}
			\sum_{\nu \in S} &\left( \nu(\sigma) - \min_{i=1, \ldots, L+1} \nu(\varphi_i) \right) \leq \\
			&\leq \bino{L+1}{2} (\abs{S} + 2\gfr_F - 2) + \deg (h(x)) \\
			&\leq \frac{1}{2} L (L+1) m_0 \left( \deg \alpha_1 + \sum_{j=1}^{d} \deg \alpha_j \right) + m_0 \deg h \\
			&\leq L (L+1) m_0 \sum_{j=1}^{d} \deg \alpha_j + m_0 \deg h.
		\end{align*}
		On the other hand we have $ \nu(\sigma) \geq \min_{i=1, \ldots, L+1} \nu(\varphi_i) $ for every valuation $ \nu $ and thus the lower bound
		\begin{align*}
			\sum_{\nu \in S} \left( \nu(\sigma) - \min_{i=1, \ldots, L+1} \nu(\varphi_i) \right)
			&\geq \sum_{P \mid \infty} \left( \nu_P(\sigma) - \min_{i=1, \ldots, L+1} \nu_P(\varphi_i) \right) \\
			&\geq \sum_{P \mid \infty} \left( \nu_P(\sigma) - \nu_P(h(x)) \right) \\
			&= \sum_{P \mid \infty} \nu_P(\sigma) - \sum_{P \mid \infty} e(P \mid \infty) \cdot \nu_{\infty}(h(x)) \\
			&= \sum_{P \mid \infty} \nu_P(\sigma) - m_0 \nu_{\infty}(h(x)) \\
			&= \sum_{P \mid \infty} e(P \mid \infty) \cdot \nu_{\infty}(\sigma) + m_0 \deg h \\
			&= m_0 \nu_{\infty}(\sigma) + m_0 \deg h \\
			&\geq m_0nJ + m_0 \deg h.
		\end{align*}
		Let us now compare the upper and lower bounds. Since $ m_0 \deg h $ appears on both sides, we can subtract it and get
		\begin{equation*}
			m_0nJ \leq L (L+1) m_0 \sum_{j=1}^{d} \deg \alpha_j.
		\end{equation*}
		Dividing by $ m_0 $ and remembering $ J = 1 $ as well as $ L \leq d $ yields
		\begin{equation}
			\label{p2-eq:gboundforn}
			n \leq  d (d+1) \sum_{j=1}^{d} \deg \alpha_j.
		\end{equation}
		
		After this we consider now the case that $ \set{h(x), t_1(x), \ldots, t_L(x)} $ is linearly dependent over $ \CC $.
		We can assume that $ \set{t_1(x), \ldots, t_L(x)} $ is linearly independent, since otherwise we are able to group them together and the first case is still working if the $ t_{h_2,\ldots,h_d}(x) $ have constant coefficients.
		This implies that in a relation of linear dependence $ h(x) $ has a nonzero coefficient. Thus there exist complex numbers $ c_i \in \CC $ such that
		\begin{equation}
			\label{p2-eq:generallasum}
			h(x) = \sum_{i=1}^{d} c_i t_{e_i}(x).
		\end{equation}
		Here we have used the above mentioned restriction on the possible elements of $ \set{t_1(x), \ldots, t_L(x)} $.
		During the following calculation we will use the abbreviation $ d_i = b_{e_i} c_i $.
		We start with equation \eqref{p2-eq:generallasum} and get
		\begin{align}
			\label{p2-eq:generallah}
			h(x) &= c_1 t_{e_1}(x) + \cdots + c_d t_{e_d}(x) \nonumber \\
			&= d_1 \alpha_1^{n/m} + d_2 \alpha_1^{n/m} \left( \frac{\alpha_2}{\alpha_1} \right)^n + \cdots + d_d \alpha_1^{n/m} \left( \frac{\alpha_d}{\alpha_1} \right)^n.
		\end{align}
		We see that $ h(x) $ is of the form $ \alpha_1^{n/m} \cdot R $ with an $ R \in \CC(x) $.
		Thus $ \alpha_1^{n/m} $ must be an element of $ \CC(x) $.
		By the definition of $ m_0 $ it follows that $ m_0 \mid n $ and therefore there exists an integer $ \ell $ such that $ n = m_0 \ell $.
		Moreover $ \alpha_1^{n/m} \in \CC(x) $ implies $ \alpha_1^{n/m} \in \CC[x] $ since $ \alpha_1 $ is a polynomial.
		Let us now rewrite equation \eqref{p2-eq:generallah} as follows
		\begin{equation}
			\label{p2-eq:mustbepoly}
			h(x) - d_1 \alpha_1^{n/m} = \frac{d_2 \alpha_1^{n/m} \alpha_2^n + \cdots + d_d \alpha_1^{n/m} \alpha_d^n}{\alpha_1^n}.
		\end{equation}
		The left hand side of equation \eqref{p2-eq:mustbepoly} is a polynomial. So the right hand side must be, too. Since the denominator has degree $ n \deg \alpha_1 $ and the numerator degree at most $ \frac{n}{m} \deg \alpha_1 + n \deg \alpha_2 < n + n \deg \alpha_2 \leq n \deg \alpha_1 $, the only possibility is that both sides are zero.
		Hence
		\begin{equation*}
			a_1 \alpha_1^n + \cdots + a_d \alpha_d^n = (h(x))^m = d_1^m \alpha_1^n
		\end{equation*}
		which is equivalent to
		\begin{equation*}
			a_2 \alpha_2^n + \cdots + a_d \alpha_d^n = (d_1^m-a_1) \alpha_1^n.
		\end{equation*}
		The left hand side is a polynomial of degree $ n \deg \alpha_2 $, but the right hand side is either zero or of degree $ n \deg \alpha_1 $. This is a contradiction.
		Therefore the linear dependent case cannot occur.
		
		Altogether we must be in the linear independent case and thus have the bound \eqref{p2-eq:gboundforn} for the index $ n $.
		Hence we only need to choose $ m > \deg \alpha_1 $ large enough such that for all $ n \leq  d (d+1) \sum_{j=1}^{d} \deg \alpha_j $ the polynomial $ G_n(x) $ is not an $ m $-th power.
	\end{proof}


\begin{thebibliography}{99}
		\bibitem{B}
			\textsc{A. Bodin},
			Decomposition of polynomials and approximate roots,
			\textit{Proc. Amer. Math. Soc.} \textbf{138} (2010), no. 6, 1989--1994.
		\bibitem{CZ}
			\textsc{P. Corvaja and U. Zannier},
			Finiteness of odd perfect powers with four nonzero binary digits,
			\textit{Ann. Inst. Fourier} \textbf{63} (2013), no. 2, 715--731.
		\bibitem{E}
			\textsc{P. Erd\H{o}s},
			On the number of terms of the square of a polynomial,
			Niew Arch. Wiskunde (2) {\bf 23} (1949), 63--65.
		\bibitem{F}
			\textsc{C. Fuchs},
			On the Diophantine equation $ G_n (x)=G_m (P(x)) $ for third order linear recurring sequences,
			\textit{Port. Math. (N.S.)} \textbf{61} (2004), no. 1, 1--24.
		\bibitem{FK}
			\textsc{C. Fuchs and C. Karolus},
			Composite polynomials in linear recurrence sequences,
			\textit{Ann. Math. Blaise Pascal}, to appear.
		\bibitem{FKK}
			\textsc{C. Fuchs, C. Karolus and D. Kreso},
			Decomposable polynomials in second order linear recurrence sequences, \textit{Manuscripta Math.} \textbf{159(3)} (2019), 321--346.
		\bibitem{FMZ}
			\textsc{C. Fuchs, V. Mantova and U. Zannier},
			On fewnomials, integral points, and a toric version of Bertini's theorem,
			\textit{J. Amer. Math. Soc.} \textbf{31} (2018), no. 1, 107--134.
		\bibitem{FP1}
			\textsc{C. Fuchs and A. Peth\H{o}},
			Effective bounds for the zeros of linear recurrences in function fields,
			\textit{J. Th\'eor. nombres Bordeaux} \textbf{17} (2005), 749--766.
		\bibitem{FP2}
			\textsc{C. Fuchs and  A. Peth\H{o}},
			Composite rational functions having a bounded number of zeros and poles, \textit{Proc. Amer. Math. Soc.} \textbf{139} (2011), no. 1, 31--38.
		\bibitem{FPT1}
			\textsc{C. Fuchs, A. Peth\H{o} and R.F. Tichy},
			On the Diophantine equation $G_n(x)=G_m(P(x))$,
			\textit{Monatsh. Math.} \textbf{137} (2002), no. 3, 173--196.
		\bibitem{FPT2}
			\textsc{C. Fuchs, A. Peth\H{o} and R.F. Tichy},
			On the Diophantine equation $G_n(x)=G_m(P(x))$: higher-order recurrences,
			\textit{Trans. Amer. Math. Soc.} \textbf{355} (2003), no. 11, 4657--4681.
		\bibitem{FPT3}
			\textsc{C. Fuchs, A. Peth\H{o} and R.F. Tichy},
			On the Diophantine equation $G_n(x)=G_m(y)$  with $Q(x,y)=0$,
			Diophantine approximation, 199--209, Dev. Math., 16, Springer Wien-New York, Vienna, 2008.
		\bibitem{FZ}
			\textsc{C. Fuchs and U. Zannier},
			Composite rational functions expressible with few terms,
			\textit{J. Eur. Math. Soc. (JEMS)} \textbf{14} (2012), 175--208.
		\bibitem{M}
			\textsc{A. Moscariello},
			Lacunary polynomials and compositions,
			PhD thesis, University of Pisa, 2019.
		\bibitem{PT}
			\textsc{A. Peth\H{o} and S. Tengely},
			On composite rational functions,
			Number theory, analysis, and combinatorics, 241--259, De Gruyter Proc. Math., De Gruyter, Berlin, 2014.
		\bibitem{R}
			\textsc{J. Rickards},
			When is a polynomial a composition of other polynomials?,
			\textit{Amer. Math. Monthly} \textbf{118} (2011), no. 4, 358--363.
		\bibitem{Sch}
			\textsc{A. Schinzel},
			On the number of terms of a power of a polynomial,
			\textit{Acta Arith.} \textbf{49} (1987), 55--70.
		\bibitem{S}
			\textsc{A. Schinzel},
			Polynomials with special regard to reducibility. With an appendix by Umberto Zannier,
			Encyclopedia of Mathematics and its Applications, 77. Cambridge University Press, Cambridge, 2000.
		\bibitem{St}
			\textsc{H. Stichtenoth},
			Function Fields and Codes,
			Universitext, Springer-Verlag, Berlin, 1993.
		\bibitem{Z2}
			\textsc{U. Zannier},
			On the integer solutions of exponential equations in function fields,
			\textit{Ann. Inst. Fourier (Grenoble)} \textbf{54} (2004), no. 4, 849--874.
		\bibitem{Z3}
			\textsc{U. Zannier},
			On the number of terms of a composite polynomial,
			\textit{Acta Arith.} \textbf{127} (2007), no. 2, 157--167.
		\bibitem{Z1}
			\textsc{U. Zannier},
			On composite lacunary polynomials and the proof of a conjecture of Schinzel,
			\textit{Invent. Math.} \textbf{174} (2008), no. 1 127--138.
		\bibitem{Z4}
			\textsc{U. Zannier},
			Addendum to the paper: \glqq On the number of terms of a composite polynomial\grqq,
			\textit{Acta Arith.} \textbf{140} (2009), no. 1, 93--99.
	\end{thebibliography}
\end{document}